%% file: main.tex
\documentclass[fleqn,5p,twocolumn]{elsarticle}

\makeatletter
\def\ps@pprintTitle{%
 \let\@oddhead\@empty
 \let\@evenhead\@empty
 \def\@oddfoot{}%
 \let\@evenfoot\@oddfoot}
\makeatother

\usepackage{defs}

\newtheorem{theorem}{Theorem}
\newtheorem{proposition}[theorem]{Proposition}

\newtheorem{lemma}[theorem]{Lemma}
\newtheorem{remark}[theorem]{Remark}

\begin{document}
\title{Cross apprenticeship learning framework: Properties and solution approaches}
\affiliation[iitb]{organization={Systems and Control Engineering, Indian Institute of Technology Bombay},
    addressline={Powai}, , 
    state={Maharashtra},
    country={India}}
\affiliation[rug]{organization={Engineering \& Technology Institute Groningen, University of Groningen}, 
    city={Groningen},
    country={The Netherlands}}
\author[iitb]{Ashwin Aravind}
\ead{ashwin@sc.iitb.ac.in}
\author[iitb]{Debasish Chatterjee}
\ead{dchatter@iitb.ac.in}
\author[rug]{Ashish Cherukuri}
\ead{a.k.cherukuri@rug.nl}
\begin{abstract}
	\input{S0_Abstract}
\end{abstract}
\maketitle
\input{S0_Introduction}
\input{S0_Preliminaries}
\input{S1_Problem_statement}
\input{S2_Performance}
\input{S3_Centralized_solution}
\input{S5_Simulations}
\input{S6_Conclusion}

\bibliographystyle{plain}
\bibliography{refs.bib}
\end{document}

%% file: S0_Abstract.tex
Apprenticeship learning is a framework in which an agent learns a policy to perform a given task in an environment using example trajectories provided by an expert. In the real world, one might have access to expert trajectories in different environments where the system dynamics is different while the learning task is the same. For such scenarios, two types of learning objectives can be defined. One where the learned policy performs very well in one specific environment and another when it performs well across all environments. To balance these two objectives in a principled way, our work presents the cross apprenticeship learning (CAL) framework. This consists of an optimization problem where an optimal policy for each environment is sought while ensuring that all policies remain close to each other. This nearness is facilitated by one tuning parameter in the optimization problem. We derive properties of the optimizers of the problem as the tuning parameter varies. Since the problem is nonconvex, we provide a convex outer approximation. Finally, we demonstrate the attributes of our framework in the context of a navigation task in a windy gridworld environment.

%% file: S0_Introduction.tex
\section{Introduction}
\label{sec:intro}
Reinforcement learning involves learning via interaction with the environment to perform a task optimally in a sequential decision-making process~\cite{RSS-AGB:18}. Commonly, the agent takes an action at a state, transitions to another state, obtains a reward from the environment and repeats the whole process again. Learning occurs when the agent looks to maximize the long-term reward, and so the efficacy of learning relies heavily on the reward structure. Poorly defined rewards lead to unwanted behaviour. In several control applications, defining appropriate rewards is difficult, and most likely, the desired behaviour can be demonstrated by an expert. For such cases, several methods under the broad umbrella of \emph{learning from demonstrations} are studied in the past, where the behaviour of an expert is available in terms of state-action trajectories~\cite{TO-JP-GN-JAB-PA-JP:18}. This information can be used in different ways, out of which, a common one is the framework of \emph{apprenticeship learning}~\cite{US-MB-RS:08}. Here, the goal is to recover optimal policies for a given Markov decision process (MDP) using demonstrations from the expert and the fact that the set where the reward function belongs to is known.

In real life, we envision scenarios where trajectories from multiple experts in different environments are available, but the underlying task is common across environments. In such settings, an agent in an environment can learn a policy that seeks a trade-off between its performance in its own environment and across multiple environments. The former is attractive when the agent is supposed to only operate in its own environment, and the available trajectories from its own expert are sufficient. The latter is advantageous in cases where the learned policy is supposed to work well across a range of environments and also possibly work as a warm start for specialized learning in any particular environment. Taking these considerations as the motivation, we present the \emph{cross apprenticeship learning} (CAL) framework in this work. We analyze the properties of the policies obtained from this framework and then address computational issues. 

\subsection*{Literature review}
Apprenticeship learning, as introduced in~\cite{PA-AN:04} consisted of two steps, first was to infer the reward function governing the expert's actions using inverse reinforcement learning, and the second was learning a suitable optimal policy for this reward function using reinforcement learning. Here, although the reward function was unknown to the agent, it was known to belong to the set of linear combinations of certain basis vectors. The applications of this framework are plenty, for example, to learn aerobatic manoeuvres on a helicopter~\cite{PA-AC-MQ-AN:06, AC-PA-AN:08}, quadrupled locomotion~\cite{JK-PA-AN:07}, navigation in a parking lot~\cite{PA-DD-AN-ST:08}, and automated parking~\cite{PF-ZY-LX-ZF-Zl-DZ:21}. In~\cite{US-RES:08} a game-theoretic approach for apprenticeship learning was proposed: the problem was cast as a two-player zero-sum game where the learning agent chooses the policy and the environment selects the reward function. This framework resulted in computationally inexpensive method and it found policies that are guaranteed to be at least as good as the expert policy for any given reward function. Building on~\cite{US-RES:08} and the linear programming (LP) approach for finding optimal policies given in~\cite{MLP:94}, the work~\cite{US-MB-RS:08} proposes an LP formulation for apprenticeship learning. The work~\cite{AK-GB-JL:19}, motivated by~\cite{YAY-PLB-XC-AM:19}, extended this LP framework to large-scale problems by solving an approximate problem where the decision variable is assumed to lie in a subspace generated by feature vectors. 

In our work, we use the LP framework for apprenticeship learning as a starting point. Our objective in this article differs from the above mentioned methods because we wish to learn a policy that is able to perform a task well in multiple environments by exploiting the availability of expert demonstrations in these environments. Such policies have a definite edge in terms of robustness as compared to policies that are learned in only one environment.

Closely related to our work are~\cite{PB-DS:19},~\cite{SB-PK-SM-CT-TZ:21} and~\cite{JC-SH-WJ-MC-SC-YS:22}. The work~\cite{PB-DS:19} aims to find a policy that performs well in different scenarios of an MDP where the scenarios are supposed to be representative of the change in the agent's environment. We note that the setting is not of learning from an expert. Instead, they assume that the reward function is given.  In~\cite{SB-PK-SM-CT-TZ:21}, expert demonstrations from different environments, parameterized by a context variable, are used to infer a parameterized reward function. The aim is to use the inferred function to perform learning in unseen environments.
The work~\cite{JC-SH-WJ-MC-SC-YS:22} explores a similar setup for imitation learning as ours. Here, minimization of the Jensen-Shannon divergence between the agent's policy and the experts' policies in different environments improved robustness to variations in environment dynamics compared to baseline imitation learning techniques. Unlike these methods, we use the LP-based approach to define cross-learning, where we borrow the key ideas of centrality of policies from~\cite{JC-JAB-MCF-AR:20} to find a middle ground between performance in one single environment and performance in all environments.

Apart from these works, there is a growing interest in inverse reinforcement or imitation learning for linear systems~\cite{HY-PS-MJ-MA:22}, nonlinear systems~\cite{ZZ-MB-NB:18, ST-AL-SH:21, ST-AR-TZ-NM:21} and MDPs~\cite{FM-AH-SN-UT:21, LR-ER:20, AK-GB-JL:21}. 
\subsection*{Setup and contributions}

For a single-agent single-environment case, the apprenticeship learning framework involves finding a policy that minimizes the worst-case discrepancy between the cost incurred by the said policy and an expert policy. In here, the worst-case discrepancy is computed by considering all cost functions that belong to a linear subspace spanned by a certain number of basis vectors. This is motivated by the setting where the learning agent does not have access to the actual cost function driving the expert behaviour but knows the set where it belongs to. This worst-case minimization problem can be cast as an LP in terms of the occupation measure. It is assumed that the learning agent does not have access to the policy of the expert. Instead, the occupation measure corresponding to the expert policy is available as it can be easily approximated using available expert trajectories. 

Our \emph{first contribution} proposes the CAL framework that extends the above defined single agent apprenticeship learning to multiple agents. At the core of this framework is the optimization problem where we seek a policy for each environment that balances two objectives. First, it minimizes the worst-case discrepancy measure, as explained above for a single agent case, for its own environment. Second, it aims to be in close proximity to policies associated to other environments. While the former is codified in the objective function of the CAL optimization problem, the latter appears as a linear constraint. The degree of proximity between policies is tuned by a parameter termed as the centrality measure. Our \emph{second contribution} is to present properties of the optimizers of the CAL problem as the centrality measure varies from low to high values. We show that when this parameter is low, all policies are close to each other and so the obtained optimizers have good \emph{generic performance}. That is, policies perform well across all environments. On the other hand, when the parameter value is high, each agent's policy maximizes performance in its own environment, that is, it displays good \emph{specific performance}. Since the CAL problem is nonconvex, our \emph{third contribution} is an outer convex approximation of the problem using McCormick envelopes.  We then discuss how this  approximation can be solved in a distributed manner. Our \emph{last contribution}  demonstrates the properties of the CAL framework in a numerical example where agents learn to navigate to a goal position in a windy gridworld. 

We organize the rest of our article as follows. Section~\ref{sec:pre} provides preliminaries. The CAL framework is presented in Section~\ref{sec:prb_stmt}. In Section~\ref{sec:performance}, we provide properties of the optimizers of the CAL optimization problem.  Section~\ref{sec:centralised} outlines a convex outer approximation of the CAL problem. Section~\ref{sec:simulation} demonstrates the use of the presented framework for a navigation task in a windy gridworld environment. 

%% file: S0_Preliminaries.tex
\section{Preliminaries}\label{sec:pre}
Here we collect notations and background on perturbation analysis of optimization problems.
\subsection{Notations}\label{sec:notation}
We use $\R{}{}$ and $\R{}{\ge 0}$ to denote real and nonnegative real numbers, respectively. Unless otherwise specified, $\norm{\cdot}$ is $\norm{\cdot}_{2}$. By $e_k$ we represent a vector of dimension $\Tenv$ with all entries being 0 expect for the $k^{\text{th}}$ entry which is 1. A vector with all entries as unity is denoted by $\ones$.  A $k$-dimensional simplex is represented by $\simplex{k} = \setdef{x \in \R{k}{\ge 0}}{\ones^\top x = 1}$. For any positive integer $n$, we use the notation $\until{n}=\{1,2,\cdots,n\}$. The number of elements in a set $\State$ is denoted by $\abs{\State}$. Given two sets $X$ and $Y$, a set-valued map $f: X \rightrightarrows Y$ associates to each point in $X$ a subset of $Y$. The set-valued map $f$ is closed if its graph  $\operatorname{gph}(f):=\setdef{(x,y)\in X\times Y}{y \in f(x) }$ is closed. Furthermore, the set-valued map $f$ is upper semicontinuous at a point $x_0 \in X$ if for any neighborhood $\Neig{f(x_0)}$ of the set $f(x_0)$ there exists a neighborhood $\Neig{x_0}$ of $x_0$ such that for every $x \in \Neig{x_0}$ the inclusion $f(x) \subset \Neig{f(x_0)}$ holds. If this property holds for all $x_0 \in X$, then $f$ is said to be upper semicontinuous. 

\input{A4_USC}

%% file: A4_USC.tex
\subsection{Perturbation of parameterized optimization problems }
\label{A:USC}
Consider the following problem:
\begin{equation}
\label{eq:par_opt}
\begin{aligned}
\min_{\dummyx} & \quad \dumff(\dummyx,\dummyu) \\
\subjectto & \quad  \dummyx\in \Dummyx,\\
 & \quad \dumfG(\dummyx,\dummyu) \leq  0,
\end{aligned}
\end{equation}
where $u$ is a parameter that belongs to a closed set $\dummyu \in \Dummyu \subset \R{n_u}{}$ and $\Dummyx \subset \R{n_x}{}$ is a closed set. The functions $\dumff:\Dummyx\times\Dummyu\to\R{}{}$ and  $\dumfG:\Dummyx\times\Dummyu\to\R{}{}$ are continuous. The feasibility set for the above optimization problem can be parameterized as the following set-valued map:
\begin{equation}
    \Dummyu \ni \dummyu \mapsto \HH(\dummyu):=\setdef{\dummyx\in \Dummyx}{\dumfG(\dummyx,\dummyu) \leq 0}.
\end{equation}
Similarly, the set of optimizers of~\eqref{eq:par_opt} is written as the following set-valued map: 
\begin{equation}
\label{eq:par_optmap}
    \Dummyu \ni \dummyu \mapsto \S(\dummyu):=\argmin_{\dummyx\in\HH(\dummyu)}\dumff(\dummyx,\dummyu).
\end{equation}
We are interested in the continuity of the map $\S:\U \rightrightarrows \Dummyx$ in the neighborhood of a point $\dummyu_0 \in \Dummyu$ (see Section~\ref{sec:pre}-\ref{sec:notation} for relevant definitions).
\begin{proposition}\longthmtitle{Upper semicontinuity of $\S$~\cite[Proposition 4.4]{JB-AS:00}}
\label{prop:usc}
Given $\dummyu_0 \in \Dummyu$, suppose the following hold:
\begin{enumerate} 
    \item the map $\HH(\cdot)$ is closed,
    \item there exists $\alpha\in\R{}{}$ and a compact set $C\subset\Dummyx$ such that for every $\dummyu$ in the neighborhood of $\dummyu_0$, the level set $\operatorname{lev}_{\alpha} \dumff(\cdot, \dummyu):=\setdef{\dummyx \in \HH(\dummyu)} {\dumff(\dummyx,\dummyu) \leq \alpha}$ is nonempty and contained in $C$, 
    \item for any neighborhood $\Neig{\S(\dummyu_0)} \subset \Dummyx$ of the set $\S(\dummyu_0)$, there exists a neighborhood $\Neig{\dummyu_0} \subset \Dummyu$ of $\dummyu_0$ such that $\Neig{\S(\dummyu_0)}  \cap \HH(\dummyu) \neq \emptyset$ for all $\dummyu \in \Neig{\dummyu_0}$.
\end{enumerate}
Then, the set-valued map $\dummyu\mapsto\S(\dummyu)$ is upper semicontinuous at $\dummyu_0$.
\end{proposition}

%% file: S1_Problem_statement.tex
\section{Problem statement}
\label{sec:prb_stmt}
We consider $\Tenv$ learning agents and their corresponding environments. Each agent $i \in \until{\Tenv}$ is associated with a Markov decision process given by the tuple $\MDP_{i}=\left(\State, \Action, \PT^{i}, \dis, \Id \right)$. Here, the finite sets $\State := (\state_1, \dots, \state_{\abs{\State}})$ and $\Action := (\action_1, \dots , \action_{\abs{\Action}})$ represent the common state and action spaces, respectively. Agents evolve in different environments specified by their individual transition matrices. In particular, the transition matrix for agent $i$ is  $\PT^i \in [0,1]^{\abs{\State} \abs{\Action} \times \abs{\State}}$, where given a state-action pair $(\state,\action) \in \State \times \Action$, the row corresponding to it, $\PT^i_{(\state,\action),:}$, gives the distribution of the next state. For notational convenience, we also denote this distribution as $\PT^i(\cdot | \state,\action)$. Thus, given a  state $\hat{\state} \in \State$, the probability of reaching it from state $\state$ using action $\action$ is $\PT^i(\hat{\state} | \state,\action)$. The discount factor and the distribution of the initial state of all agents are denoted by $\dis \in (0,1)$ and $\Id \in \simplex{\abs{\State}}$, respectively.

An agent $i$ has access to an expert's behaviour in its environment. Each expert $i$ acts according to a policy given by the map $\map{\policy{}{E_i}}{\State}{\simplex{\abs{\Action}}}$, that is, at state $s \in \State$, the distribution of the selected action by the expert is given by $\policy{}{E_i}(s)$. We assume that each expert's policy is stationary and we denote the set of stationary policies by $\polspace$, that is, $\polspace := \{\policy{}{}:\State\to\simplex{\abs{\Action}}\}$. The expert $i$'s policy is aimed at minimizing a cost function $\Cst_i:\State\times\Action \rightarrow \R{}{}$ associated with the task. This cost is unknown to us, however, the set where the cost belongs is known and is given as $\Cst_i \in \CB := \{ \sum_{j=1}^{\CBn} \CW_j \CBe_j \mid \norm{ \CW }_\infty \leqslant 1\}$, where each $\CW_j$ is the $j$-th component of the cost weight vector $\CW \in \R{n_c}{}$ and it represents the weight associated to the $j$-th basis vector $\CBe_j \in \R{\abs{\State}\abs{\Action}}{}$. These $\CBn$ basis vectors are fixed and they satisfy $\norm{\CBe_j}_\infty \leqslant 1$ for all $j \in \until{\CBn}$. The behavior of the expert that is governed by its policy is known to us through the occupation measure that it generates. We elaborate on this next.

Given a policy $\policy{}{}$ and the initial distribution $\Id \in \simplex{\abs{\State}}$, the induced probability measure over the canonical sample space $\Omega := (\State \times \Action)^\infty$ for agent $i$ is, $\PP_{\Id}^{\policy{}{}, i}[\cdot]$. Here, $\PP_{\Id}^{\policy{}{}, i}[\state_t=\state, \action_t=\action]$ denotes the probability that agent $i$ is in state $\state$ and takes an action $\action$ at time instant $t$ starting from an initial state distribution $\Id$ and following a policy $\policy{}{}$. For a given $\policy{}{} \in \polspace$, the discounted occupation measure for agent $i$, denoted $\Ocm{\policy{}{}}{i} :\State \times \Action \to \R{}{}$, is defined as $\Ocm{\policy{}{}}{i} (\state,\action) := \sum_{t=0}^{\infty} \dis^t \PP_{\Id}^{\policy{}{},i}[\state_t=\state, \action_t=\action]$. It is interpreted as the discounted expected number of times a state-action pair is visited by the agent $i$ starting from an initial state distribution $\Id$, and by following a policy $\policy{}{}$. We assume that for each environment $i$ the occupation measure generated by the expert $\Ocm{\policy{}{E_i}}{i}$ is known. This constitutes the behavior of the expert available to us. For environment $i$, consider the set
\begin{equation}
\label{eq:ocm_fes}
	\F{i} := \setdef{\Ocm{}{}\in\R{\abs{\State}\abs{\Action}}{\geq 0}} {(\B-\dis\PT^i)^\top \Ocm{}{}=\Id},
\end{equation}
where $\B \in \{0,1\}^{|\State||\Action| \times |\State|}$ is a binary matrix where the element $\B_{(\state_j,\action_k),\state_l}=1$ if $j=l$, and $\B_{(\state_j,\action_k),\state_l}=0$ otherwise. From~\cite[Theorem 2]{US-MB-RS:08} we know that, for every $\policy{}{} \in \polspace$, the corresponding occupation measure $\Ocm{\policy{}{}}{i}$ belongs to $\F{i}$. Also, given any $\Ocm{}{i} \in \F{i}$, a stationary policy $\policy{}{\Ocm{}{i}} \in \polspace$ is obtained by setting $\policy{}{\Ocm{}{i}}(\action, \state) := \frac{\Ocm{}{i}(\state, \action)}{\sum_{\state'\in \State}\Ocm{}{i}(\state,\action)}$. In addition, this correspondence is one-to-one, that is, the induced occupation measure for the policy $\policy{}{\Ocm{}{i}}$ is $\Ocm{\policy{}{\Ocm{}{i}}}{i}=\Ocm{}{i}$.
Given any cost function $\Cst_i:\State\times\Action \rightarrow \R{}{}$, the expected discounted cost incurred by the agent $i$ is $ \Ted_{\Cst_i} (\policy{}{})  = \E_{\Id}^{\policy{}{},i} \left[ \sum_{t=0}^{\infty} \dis^t \Cst_i(\state_t,\action_t)\right]$, here the expectation is with respect to the distribution $ \PP_{\Id}^{\policy{}{}, i}$. This can also be represented  as the inner product of discounted occupation measure and the cost vector, that is, $\Ted_{\Cst_i} (\policy{}{}) =  \langle \Ocm{\policy{}{}}{i}, \Cst_i \rangle$.

For agent $i$, the goal of learning (when decoupled from the other agents and environments) is to find a policy $\policy{}{i} \in \polspace$ such that $\langle \Ocm{\policy{}{i}}{i}, \Cst_{i}\rangle \leq \langle \Ocm{\policy{}{E_i}}{i}, \Cst_{i}\rangle $, where $\Ocm{\policy{}{E_i}}{i}$ is the occupation measure induced by the expert's policy $\policy{}{E_i}$. However, note that the expert's cost $\Cst_{i}$ is usually unknown as only the behavior in terms of trajectories is available. Instead of knowing the exact cost, we assume that the agent knows the set $\CB$ where the true cost belongs. Consequently, the goal of learning then translates to finding a policy $\policy{}{i}$ such that $\langle \Ocm{\policy{}{i}}{i}, c\rangle \leq \langle \Ocm{\policy{}{E_i}}{i}, \Cst\rangle $ for all $\Cst \in \CB$, i.e., the policy $\policy{}{}$ must out-perform the expert policy for all $\Cst \in \CB$. Such a framework is well studied in the apprenticeship learning literature, see e.g.,~\cite{PA-AN:04},~\cite{US-RES:08}, and~\cite{US-MB-RS:08}. Thus, the objective for agent $i$ in apprenticeship learning, decoupled from all other agents and environments, is: 
\begin{equation}\label{eq:decoupled_obj}
    \min_{\policy{}{i}\in\polspace}\sup_{\Cst \in \CB}(\langle \Ocm{\policy{}{i}}{i}, c\rangle - \langle \Ocm{\policy{}{E_i}}{i}, c\rangle).
\end{equation}
One can simplify the objective function~\eqref{eq:CrossLearning_obj_1-obj} by utilizing the structure of $\CB$. Following the notation in~\cite{AK-GB-JL:19}, we define $\CBE := [\CBe_1, \dots, \CBe_\CBn] \in \R{\abs{\State}\abs{\Action}\times \CBn}{}$ as the \emph{cost basis matrix}.
For every $\policy{}{i}$, the following holds~\cite[Lemma 1]{AK-GB-JL:19},
\begin{equation*}
	\sup_{\Cst\in \CB} (\langle \Ocm{\policy{}{i}}{i},\Cst\rangle-\langle \Ocm{\policy{}{E_i}}{i},\Cst \rangle) = \norm{\CBE^\top \Ocm{\policy{}i}{i} - \CBE^\top  \Ocm{\policy{}{E_i}}{i}}_1.
\end{equation*}
Thus, problem~\eqref{eq:decoupled_obj} can be equivalently written as
\begin{equation}\label{eq:decoupled_obj-reform}
	\min_{\policy{}{i}\in\polspace} \norm{\CBE^\top \Ocm{\policy{}{i}}{i} - \CBE^\top  \Ocm{\policy{}{E_i}}{i}}_1.
\end{equation}
Note that the objective function is nonnegative and the optimal value is zero as $\policy{*}{i}=\policy{}{E_i}$ is one of the optimizers. Additionally, for any optimizer $\policy{*}{i}$ of~\eqref{eq:decoupled_obj-reform}, we have
\begin{align*}
	\langle \Ocm{\policy{*}{i}}{}, \Cst\rangle = \langle  \Ocm{\policy{}{E_i}}{} , \Cst\rangle \quad \text{for all } \Cst\in \CB. 
\end{align*} 
If $\CB$ contains all possible cost functions, then the expert policy is the only optimizer of~\eqref{eq:decoupled_obj-reform}. The lower the number of basis vectors in $\CB$ the more flexibility we have to find a policy that performs as well as the expert policy $\policy{}{E_i}$.

Each agent $i$ can solve problem~\eqref{eq:decoupled_obj-reform} and obtain an optimal policy that performs well in its own environment. Such a policy might not perform well in other environments, while the learning task is same in all environments. To capture these commonalities between the environments, motivated by~\cite{JC-JAB-MCF-AR:20}, we define the following \emph{cross apprenticeship learning} (CAL) problem:
\begin{subequations}\label{eq:CrossLearning_obj_2}
\begin{align}
\min _{\{\policy{}{i}\}_{i=1}^\Tenv, \policy{}{c}} & 
\quad \sum_{i=1}^{\Tenv} \norm{\CBE^\top \Ocm{\policy{}{i}}{i} - \CBE^\top \Ocm{\policy{}{E_i}}{i}}_1 \label{eq:CrossLearning_obj_1-obj}
\\
\text{subject to} & \quad  \policy{}{c} \in \polspace,\\
&\quad \policy{}{i} \in \polspace \text{ for all } i\in\until{\Tenv},\\
& \quad \norm{\policy{}{i}-\policy{}{c}}_{\infty} \leq \centrality  \text{ for all }i \in \until{\Tenv}. \label{eq:CrossLearning_obj_1-cons}
\end{align}
\end{subequations}
We denote the set of optimizers of the above problem by $\SS{\Cal}{} \subset \polspace^{\Tenv+1}$. In the above problem, through the decision variable $\policy{}{i}$ we seek a policy that performs well in environment $i$. The objective function is decoupled in this set of \emph{individual policies}. On the other hand, these individual policies are required to be close to a \emph{cross-learned} policy $\policy{}{c}$. The variable $\centrality$ defines this proximity and is termed as the \emph{centrality measure}. The individual policies an agent learns via cross-learning sacrifices optimality in its environment for generalization across all other environments. 

Our aim in this paper is to analyze the properties of the CAL framework~\eqref{eq:CrossLearning_obj_2} and design methods to solve this optimization problem approximately. 

%% file: S2_Performance.tex
\section{Properties of CAL framework}
\label{sec:performance}
The objective of this section is to analyze the performance of the individual and the cross-learned policies across different environments. We consider the following general \emph{performance function} for a policy $\policy{}{} \in \polspace$:
\begin{equation}
\label{eq:perfrmnc}
	\performance{\perfvec}{\policy{}{}} := \sum_{i=1}^{\Tenv} \perfvec_i \norm{ \CBE^\top \Ocm{\policy{}{}}{i} - \CBE^\top \Ocm{\policy{}{E_i}}{i}}_1,
\end{equation}
where $\perfvec \in \simplex{\Tenv}$ represents the weight given to individual environments. In the above definition, the value $\norm{ \CBE^\top \Ocm{\policy{}{}}{i} - \CBE^\top \Ocm{\policy{}{E_i}}{i}}_1$ determines how well the policy $\policy{}{}$ performs in an environment $i$. A lower value indicates that the cost incurred by the policy is close to that by the expert. Therefore, a low value of performance function implies that the policy performs better across environments, where the importance attached to each environment is represented by the weighing $\beta$. Below we will analyze the properties of the above function.
\subsection{Continuity of $V_\beta$}
The right-hand side of~\eqref{eq:perfrmnc} depends on the policy implicitly through the occupation measure generated.  Therefore we will first examine the maps representing the correspondence between the policy and the occupation measure. To this end, we define the following two maps between the policy space $\polspace$ and the set of feasible occupation measures $\F{i}$ (see~\eqref{eq:ocm_fes}) for some environment $i\in\until{\Tenv}$:
\begin{subequations}
\label{eq:maps}
\begin{flalign}
\OcmPol{}{i} : &\F{i} \to \polspace, \text{ where } \, \, \OcmPol{}{i} (\Ocm{}{}) (\state,\action) = \frac{\Ocm{}{}(\state,\action)}{\sum_{\action' \in \Action} \Ocm{}{}(\state,\action')} && \label{eq:OcmPol} \\
\PolOcm{}{i} : &\polspace \to \F{i}, \text{ where } \, \,\PolOcm{}{i}(\pi)(\state,\action)&& \nonumber \\
&\qquad \qquad \qquad \quad  =\sum_{t=0}^{\infty} \dis^t \PP_{\Id}^{\policy{}{},i}[\state_t=\state, \action_t=\action],&& \label{eq:PolOcm}
\end{flalign}
\end{subequations}
for all $\Ocm{}{}\in\F{i}$, $\policy{}{} \in \polspace$ and $(\state,\action) \in \State \times \Action$. Note that $\OcmPol{}{i}$ is same for all environments. We have used the subscript to denote that the domain is different for each of these function. Recall that in the shorthand notation that we introduced earlier, we use $\OcmPol{}{i}(\Ocm{}{}) = \policy{}{\Ocm{}{}}$ and $\PolOcm{}{i} (\policy{}{}) = \Ocm{\policy{}{}}{i}$. Before we delve into analyzing properties of the above defined maps, we derive the following bounds on the occupation measure that will be used later. 
\input{S2_Perform_sup}
The Lipschitz property of the map $\PolOcm{}{i}$ established in the above result aids us in showing the same for the performance function $V_\beta$ given in~\eqref{eq:perfrmnc}. The next result formalizes this statement. This property implies that if two policies are close to each other, as might be the case due to the centrality constraint~\eqref{eq:CrossLearning_obj_1-cons} in the CAL problem, then their performance across the environments will be similar.
\begin{lemma}\longthmtitle{Sensitivity of the performance function with respect to policies}
\label{lem:pergap}
Given two policies $\policy{}{1}, \policy{}{2} \in \polspace$ and $\beta \in \simplex{\Tenv}$, the following holds: 
\begin{equation}
    \abs{\performance{\perfvec}{\policy{}{1}}-\performance{\perfvec}{\policy{}{2}}}\leq \CBn L^{\PolOcm{}{}} \sqrt{\abs{\State}\abs{\Action}} \norm{\policy{}{1}-\policy{}{2}}_2
\end{equation}
where $L^{\PolOcm{}{}} = \max_{i \in \until{\Tenv}} L_i^{\PolOcm{}{}}$, with each $L_i^{\PolOcm{}{}}$ being the Lipschitz constant for the map $\PolOcm{}{i}$ as stated in Lemma~\ref{lem:map_prop}. 
\end{lemma}
\begin{proof}
We compute 
\begin{flalign*}
    \abs{\performance{\perfvec}{\policy{}{1}}-\performance{\perfvec}{\policy{}{2}}}&=\Bigg|\sum_{k=1}^{\Tenv} \perfvec_k\bigg( \norm{ \CBE^\top \Ocm{\policy{}{1}}{k} - \CBE^\top \Ocm{\policy{}{E_k}}{k}}_1 &&\\
    &\quad- \norm{ \CBE^\top \Ocm{\policy{}{2}}{k} - \CBE^\top \Ocm{\policy{}{E_k}}{k}}_1\bigg)\Bigg| &&\\
    &\overset{(a)}{\leq} \sum_{k=1}^{\Tenv} \perfvec_k\bigg(\norm{ \CBE^\top \Ocm{\policy{}{1}}{k} - \CBE^\top \Ocm{\policy{}{2}}{k}}_1 \bigg) &&\\
    &\overset{(b)}{\leq} \norm{\CBE^\top}_{1,1}\sum_{k=1}^{\Tenv} \perfvec_k\bigg(\norm{\Ocm{\policy{}{1}}{k} - \Ocm{\policy{}{2}}{k}}_1\bigg) &&\\
    &\overset{(c)}{\leq} \CBn L^{\PolOcm{}{}} \sqrt{\abs{\State}\abs{\Action}}\norm{\policy{}{1} - \policy{}{2}}_2 \sum_{k=1}^{\Tenv} \perfvec_k &&\\
    &\overset{(d)}{=} \CBn L^{\PolOcm{}{}} \sqrt{\abs{\State}\abs{\Action}}\norm{\policy{}{1} - \policy{}{2}}_2,&&
\end{flalign*}
where inequality $(a)$ is a consequence of the triangle inequality, $(b)$ is due to the submultiplicity of induced matrix norm, $(c)$ is due to Lemma~\ref{lem:map_prop4} and the fact that the elements of the cost basis $\CBE$ satisfy $\norm{\CBe_j}_\infty \leq 1$ for all $j=1,\dots ,\CBn$, and $(d)$ is because $\beta \in \simplex{\Tenv}$. 
\end{proof}
From the above result, by considering $\perfvec = e_i$ for some environment $i \in \until{\Tenv}$, we obtain the bound on the difference in the performance of two policies in that environment. This set of Lemmas will be useful in the subsequent section in analyzing the specific and generic performance of the policies obtained through the CAL problem.
\subsection{Specific and generic performance of CAL}
As mentioned earlier, the solution of the CAL-framework results in $\Tenv$ individual policies and a cross-learned policy. In this section, we investigate the performance of these policies in individual environments as well as across environments. We term these properties as \emph{specific} and \emph{generic} performance, respectively. We demonstrate how by tuning the centrality measure $\centrality$, one targets to maximize for one of these performances. 

We first introduce relevant notation. Let the set of optimal policies for the decoupled learning problem of agent $i$ given in~\eqref{eq:decoupled_obj} be denoted as $\SS{\dec,*}{i} \subset \polspace$. That is, 
\begin{align*}
	\SS{\dec}{i} := \argmin_{\policy{}{i}\in\polspace} \norm{\CBE^\top \Ocm{\policy{}{i}}{i} - \CBE^\top \Ocm{\policy{}{E_i}}{i}}_{1}.
\end{align*}
We refer to $\policy{\dec,*}{i} \in \SS{\dec}{i}$ as the optimal \emph{decoupled policy}. Note that if we choose $\epsilon$  to be large enough, then the CAL framework finds these decoupled optimal policies in the form of individual policies. Specifically, if
\begin{align*}
	2 \centrality > \max_{i,j\in\until{\Tenv}} \setdef{\norm{ \policy{}{i} - \policy{}{j}}_\infty}{\policy{}{i} \in \SS{\dec}{i}, \policy{}{j} \in \SS{\dec}{j}}  ,
\end{align*}
then any optimizer of CAL, denoted $(\{\policy{*}{i}\}_{i=1}^{\Tenv},\policy{*}{c}) \in \SS{\Cal}{}$, satisfies $\policy{*}{i} \in \SS{\dec}{i}$ for all $i \in \until{\Tenv}$. That is, not all constraints of the CAL problem are binding. On the other hand, when $\centrality = 0$, then  $\policy{*}{i} \in \SS{\cen}{}$ for all $i \in \until{\Tenv}$ where 
\begin{align}
\label{eq:S_cen}
	\SS{\cen}{} := \argmin_{\policy{}{} \in \polspace} \sum_{i=1}^{\Tenv} \norm{\CBE^\top \Ocm{\policy{}{}}{i} - \CBE^\top \Ocm{\policy{}{E_i}}{i}}_1. 
\end{align}
We refer to any policy $\policy{\cen,*}{} \in \SS{\cen}{}$ as the optimal \emph{centralized policy}. We have the following first result that provides the specific performance of the optimizers in $\SS{\Cal}{}$.
\begin{proposition}\longthmtitle{Specific performance of $\SS{\dec}{i}$  and $\SS{\Cal}{}$}\label{le:ind-env-per}
	For any $\policy{\dec,*}{i} \in \SS{\dec}{i}$ and $(\{\policy{*}{i}\}_{i=1}^{\Tenv}, \policy{*}{c}) \in \SS{\Cal}{}$, the following hold:
\begin{enumerate}
	\item $\performance{e_i}{\policy{\dec,*}{i}} \leq \performance{e_i}{\policy{*}{i}}  \leq \performance{e_i}{\policy{*}{c}}$, where $e_i$ is the unit vector with $i$-th component being unity,
	\item $\performance{e_i}{\policy{*}{i}} \leq \performance{e_i}{\policy{*}{j}}$ for all $i$ and $j \not = i$.
\end{enumerate}
\end{proposition}    
\begin{proof}
The first inequality follows from the fact $\policy{\dec,*}{i}$ is an optimizer of $\policy{}{} \mapsto \performance{e_i}{\policy{}{}}$ over the set $\polspace$ and $\policy{*}{i}$ belongs to the set $\polspace$. For the second inequality, note that 
\begin{align}\label{eq:opt_ind}
	\policy{*}{i} \in \argmin_{\policy{}{} \in \polspace} \setdef{\performance{e_i}{\policy{}{}}}{ \norm{\policy{}{}-\policy{*}{c}}_{\infty} \leq \centrality }.
\end{align}
This is true because otherwise we contradict the fact that $(\{\policy{*}{i}\}_{i=1}^{\Tenv}, \policy{*}{c})$ is an optimizer of~\eqref{eq:CrossLearning_obj_2}. From the above observation and the fact that $\policy{*}{c}$ trivially belongs to the feasibility set of~\eqref{eq:opt_ind}, we conclude the second inequality. The last inequality also follows from the fact that $\policy{*}{j}$ belongs to the feasibility set of~\eqref{eq:opt_ind}, therefore $\performance{e_i}{\policy{*}{i}}$ is at most equal to $\performance{e_i}{\policy{*}{j}}$.  
\end{proof}
The above result shows that, as expected, the decoupled policy of environment $i$ outperforms the individual and the cross-learned policy obtained from the CAL problem in that environment. Moreover, in environment $i$, the individual optimal policy $\policy{*}{i}$ obtained in CAL performs better than the cross-learned policy  $\policy{*}{c}$ and any other individual policy $\policy{*}{j}$. The above result is irrespective of the value of the centrality measure. Next, we analyze the performance of the policies obtained across environments, where the selection of centrality measure $\centrality$ becomes key. We will use Lemma~\ref{lem:pergap} and show that for small enough values of $\centrality$ the individual optimal policies of CAL outperform the decoupled optimal policies across environments. In order to obtain the formal result, the first step is to analyze the set-valued map that gives the set of optimizers of the CAL problem given the parameter $\centrality$.
\begin{lemma}\longthmtitle{Upper semicontinuity of set of optimizers of CAL with respect to $\centrality$}
\label{lem:usc}
Define the map,
\begin{flalign}
\label{eq:phi_def}
    \fesx(\centrality):=\setdefbig{(\{\policy{}{i}\}_{i=1}^{\Tenv}&,\policy{}{c}) \in \polspace^{\Tenv+1}} {&& \nonumber \\ & \norm{\policy{}{i}-\policy{}{c}}_{\infty}\leq \centrality, \text{ for all } i \in\until{\Tenv}}&&
\end{flalign}
that gives the feasibility set of~\eqref{eq:CrossLearning_obj_2} for a given $\centrality \ge0$.
Then, the set-valued map $\SS{\Cal}{}:[0,1] \rightrightarrows \polspace^{\Tenv+1}$ defined as
\begin{equation}
\label{eq:mul_S}
    \SS{\Cal}{}(\centrality):=\argmin_{(\{\policy{}{i}\}_{i=1}^{\Tenv},\policy{}{c})\in \fesx(\centrality)} \sum_{i=1}^{\Tenv} \norm{\CBE^\top \Ocm{\policy{}{i}}{i} - \CBE^\top \Ocm{\policy{}{E_i}}{i}}_1
\end{equation}
is upper semicontinuous at $\centrality=0$.
\end{lemma}
\begin{proof}
Our proof is based on Proposition~\ref{prop:usc} that analyzes the continuity of optimizers of a parameterized optimization problem. Drawing the parallelism between~\eqref{eq:par_optmap} and~\eqref{eq:mul_S}, the decision variable $x$, the parameter $u$, the set $\Dummyu$, the objective function $f$, and the set-valued map $\HH$ as given in~\eqref{eq:par_optmap} are to be considered analogously in~\eqref{eq:mul_S} as  the variable $(\{\policy{}{i}\}_{i = 1}^\Tenv , \policy{}{c})$, the parameter $\centrality$, the set $[0,1]$, the objective function
\begin{align*}
    \bar{\dumff}(\{\policy{}{i}\}_{i=1}^{\Tenv},\policy{}{c}):=\sum_{i=1}^{\Tenv} \norm{\CBE^\top \Ocm{\policy{}{i}}{i} - \CBE^\top  \Ocm{\policy{}{E_i}}{i}}_1,
\end{align*}
and the map $\fesx$ defined in~\eqref{eq:phi_def}, respectively. Note that in~\eqref{eq:mul_S}, the objective function does not depend on the parameter. The proof now proceeds by checking the conditions of Proposition~\ref{prop:usc}. Firstly, the objective function $\bar{\dumff}$ and the constraint function in~\eqref{eq:phi_def} are continuous. The set-valued map $\fesx$ is closed. The second condition in Proposition~\ref{prop:usc} holds as $\bar{\dumff}$ is bounded on $\polspace^{\Tenv+1}$ and the set $\fesx(\centrality)$ is nonempty and contained in the compact set $\polspace^{\Tenv+1}$ for any nonnegative $\epsilon$. Lastly, for the third condition, note that for every neighborhood $\Neig{0} \subset [0,1]$ of $\centrality = 0$, we have $\fesx(0)\subset\fesx(\centrality)$ for all $\centrality\in\Neig{0}$. Consequently, $\SS{\Cal}{}(0) \subset \fesx(\centrality)$ for all $\centrality\in\Neig{0}$ and so for any neighborhood $\Neig{\SS{\Cal}{}(0)}$ of $\SS{\Cal}{}(0)$ we have $\Neig{\SS{\Cal}{}(0)} \cap \fesx(\centrality)$ for all $\centrality\in\Neig{0}$. Thus, condition three in Proposition~\ref{prop:usc} holds and so, we conclude that $\SS{\Cal}{}$ is upper semicontinuous at the origin.
\end{proof}

With the above continuity property in mind, we next show that if $\centrality$ is small, then the individual optimal policy obtained in CAL has better performance across environments as compared to the decoupled optimal policy.
\begin{proposition}\longthmtitle{Generic performance of $\SS{\dec}{i}$  and $\SS{\Cal}{}$}\label{le:acc-env-per}
For any $(\{\policy{*}{i}\}_{i=1}^{\Tenv}, \policy{*}{c}) \in \SS{\Cal}{}$ and $\policy{\cen,*}{} \in \SS{\cen}{}$, we have
\begin{equation*}
	\performance{N^{-1}\ones}{\policy{\cen,*}{}} \leq \performance{N^{-1}\ones}{\policy{*}{i}}
\end{equation*}
where $N^{-1}\ones$ denotes the vector with each entry as $\frac{1}{\Tenv}$. Further, for any $j \in [\Tenv]$, if $\performance{N^{-1}\ones}{\policy{\cen,*}{}} < \performance{N^{-1}\ones}{\policy{\dec,*}{j}}$, then there exists an $\centrality > 0$ such that the following holds,
\begin{equation*}
    \performance{N^{-1}\ones}{\policy{*}{j}} \leq \performance{N^{-1}\ones}{\policy{\dec,*}{j}} \, \, \text{ for all } (\{\policy{*}{j}\}_{j=1}^{\Tenv}, \policy{*}{c}) \in \SS{\Cal}{}(\epsilon).
\end{equation*}
\end{proposition}
\begin{proof}
The first inequality trivially follows from the definition of $\policy{\cen,*}{}$. For the second inequality, from Lemma~\ref{lem:pergap}, the function $\performance{N^{-1}\ones}$ is Lipschitz continuous everywhere on the compact set $\polspace$. Thus, for any $i \in [\Tenv]$ where $\performance{N^{-1}\ones}{\policy{\cen,*}{}} < \performance{N^{-1}\ones}{\policy{\dec,*}{i}}$ holds, there exists a neighborhood $\Neig{\SS{\cen}{}}\subset \polspace$ of the set $\SS{\cen}{}$ such that
\begin{align}\label{eq:nscen}
    \performance{N^{-1}\ones}{\policy{}{}} \leq \performance{N^{-1}\ones}{\policy{\dec,*}{i}}, \quad \text{ for all } \policy{}{} \in \Neig{\SS{\cen}{}}.
\end{align}
Noting the fact that $\SS{\Cal}{}(0) = (\SS{\cen}{})^{N+1}$ and using Lemma~\ref{lem:usc}, we conclude that there exists $\bar{\centrality}>0$ such that for all $\centrality \in [0,\bar{\centrality})$ we have $\SS{\Cal}{}(\centrality) \subset (\Neig{\SS{\cen}{}})^{N+1}$. This inclusion along with the inequality~\eqref{eq:nscen} yields the conclusion.
\end{proof}
The two results presented in this section highlight the fact that the optimizers of the CAL framework balance the properties of the centralized and decoupled optimal policies. This balance is tunable using the centrality measure $\centrality$. Moreover, we believe that the framework also allows for agents to learn from each other's experts. We wish to explore these ideas in future. 

%% file: S2_Perform_sup.tex
\begin{lemma}
\label{lem:ocm_sum}
\longthmtitle{Bounds on the state occupation measure}
Given a policy $\policy{}{} \in \polspace$ and any environment $i\in\until{\Tenv}$, for all $\state\in\State$, it follows that $\Id(\state) \leq \sum_{\action\in\Action}\Ocm{\policy{}{}}{i}(\state, \action) \leq \frac{\abs{\Action}}{1-\dis}$.
\end{lemma}
\begin{proof}
Following the definition of the occupation measure, we have,
\begin{equation*}
\begin{aligned}
    \sum_{\action\in\Action}\Ocm{\policy{}{}}{i}(\state, \action)&=\sum_{\action\in\Action} \sum_{t=0}^{\infty} \dis^t \PP_{\Id}^{\policy{}{},i}[\state_t=\state, \action_t=\action] 
    \\
    &= \sum_{t=0}^{\infty} \dis^t \sum_{\action\in\Action} \PP_{\Id}^{\policy{}{},i}[\state_t=\state, \action_t=\action] \\ 
    & = \sum_{t=0}^{\infty} \dis^t \PP_{\Id}^{\policy{}{},i}[\state_t=\state]
    \\
    &\overset{(a)}{=} \PP_{\Id}^{\policy{}{},i}[\state_0=\state] + \sum_{t=1}^{\infty} \dis^t \PP_{\Id}^{\policy{}{},i}[\state_t=\state]\overset{(b)}{\geq} \Id(\state),
\end{aligned}    
\end{equation*}
where $\PP_{\Id}^{\policy{}{},i}[\state_t=\state]$ is the probability that the agent $i$ is in state $s$ at time instant $t$ starting with initial distribution $\Id$ and following policy $\policy{}{}$. In the above relations, (a) is obtained by taking out the first term from the summation and (b) is due to that fact that the second term is nonnegative in the previous equality and $\Id(\state) = \PP_{\Id}^{\policy{}{},i}[\state_0=\state] $. For the upper bound we have,
\begin{flalign*}
    \sum_{\action\in\Action}\Ocm{\policy{}{}}{i}(\state, \action)= \sum_{\action\in\Action} \sum_{t=0}^{\infty} \dis^t \PP_{\Id}^{\policy{}{},i}[\state_t&=\state, \action_t=\action] &&\\
    &\leq \sum_{\action\in\Action}\sum_{t=0}^{\infty} \dis^t=\frac{\abs{\Action}}{1-\dis}.&&
\end{flalign*}
This completes the proof. 
\end{proof}
Using the preceding results, in the following Lemma we present continuity properties of the maps $\OcmPol{}{i}$ and $\PolOcm{}{i}$. In particular, both these functions are bijections, continuously differentiable, and Lipschitz.
\begin{lemma}
\longthmtitle{Properties of the maps $\OcmPol{}{i}$ and $\PolOcm{}{i}$}
\label{lem:map_prop}
For some environment $i \in \until{\Tenv}$, consider the maps $\OcmPol{}{i}$ and $\PolOcm{}{i}$ as defined in~\eqref{eq:maps}. The following properties hold:
\begin{enumerate}
    \item \label{lem:map_prop1} Maps $\OcmPol{}{i}$ and $\PolOcm{}{i}$ are continuously differentiable on $\F{i}$ and $\polspace$, respectively.
    \item \label{lem:map_prop3} For all $\Ocm{}{1}, \Ocm{}{2}\in \F{i}$ we have, $\norm{\OcmPol{}{i} (\Ocm{}{1})-\OcmPol{}{i}(\Ocm{}{2})}_{2} \leq \frac{2}{\min_{\state\in\State}\Id(\state)} \norm{\Ocm{}{1}-\Ocm{}{2}}_1$.
    \item \label{lem:map_prop4} There exists a $L_{i}^{\PolOcm{}{}}>0$ such that, for all $\policy{}{1}, \policy{}{2} \in \polspace$ we have, $\norm{ \PolOcm{}{i}(\policy{}{1})- \PolOcm{}{i}(\policy{}{2})}_{2}\leq L_{i}^{\PolOcm{}{}} \norm{\policy{}{1}-\policy{}{2}}_{2}$.
\end{enumerate}
\end{lemma}
\begin{proof}
The map $\PolOcm{}{i}$ has Lipschitz continuous gradient over the set $\Pi$, as shown in~\cite[Proposition 1]{JZ-AK-ASB-CS-MW:20}. Thus, $\PolOcm{}{i}$ is continuously differentiable. Regarding $\OcmPol{}{i}$, denote the Jacobian as the map 
$D \OcmPol{}{i} : \F{i} \to \R{\abs{\State}\abs{\Action} \times \abs{\State}\abs{\Action}}{}$. For a given $\Ocm{}{}$, the $(i,j)$-th element of the Jacobian $D \OcmPol{}{i}(\Ocm{}{})$, where index $i$ and $j$ correspond to state-action pairs $(\state_i,\action_i)$ and $(\state_j,\action_j)$, respectively, is
\begin{flalign*}
    D \OcmPol{}{i} (\Ocm{}{}) &(i,j)&&\\ &= \begin{cases} \frac{\sum_{\action'\in\Action}\Ocm{}{}(\state_i, \action')-\Ocm{}{}(\state_i, \action_i)}{(\sum_{\action'\in\Action}\Ocm{}{}(\state_i, \action'))^2} & \quad \text{if }  (\state_i,\action_i) = (\state_j,\action_j),
    \\
    \frac{-\Ocm{}{}(\state_i, \action_i)}{(\sum_{\action'\in\Action}\Ocm{}{}(\state_i, \action'))^2} & \quad \text{if } \state_i = \state_j,
    \\
    0 & \quad \text{otherwise}.
    \end{cases}&&
\end{flalign*}
From Lemma~\ref{lem:ocm_sum} we know that for any state $\state \in \State$, we have $\sum_{\action'\in\Action}\Ocm{}{}(\state, \action')\geq\Id(\state)>0$. Thus, $D \OcmPol{}{i}$ given in the above expression is well-defined and continuous on $\F{i}$. This proves the first claim. The second claim was established in~\cite[Proposition 1]{JZ-AK-ASB-CS-MW:20}. The last conclusion follows from the facts that  $\PolOcm{}{i}$ has a Lipschitz continuous gradient and it is continuously differentiable in $\polspace$.
\end{proof}

%% file: S3_Centralized_solution.tex
\section{Algorithms for solving CAL problem}
\label{sec:centralised}
In this section, we investigate both centralized and distributed approaches to approximate the solution of the CAL problem~\eqref{eq:CrossLearning_obj_2}. To this end, observe that the objective function in~\eqref{eq:CrossLearning_obj_2} is non-convex with respect to the policies but is convex with respect to the corresponding induced discounted occupation measures. Therefore, in line with the approach used in~\cite{AK-GB-JL:19}, we proceed to rewrite~\eqref{eq:CrossLearning_obj_2} in terms of the discounted occupation measure. This process results into a convex objective, but renders the constraints bilinear, as explained below. We handle the nonconvexity caused by such constraints by forming convex outer approximation of the feasibility set.

Recalling the set of feasible occupation measures given in~\eqref{eq:ocm_fes} and the bijection between policies and occupation measures, we rewrite~\eqref{eq:CrossLearning_obj_2} equivalently as 
\begin{subequations}\label{eq:CrossLearning_obj}
\begin{align}
  	\min_{\{\Ocm{}{i}\}_{i=1}^N,\policy{}{c}} & \quad \sum_{i=1}^{\Tenv} \norm{\CBE^\top\Ocm{}{i} -\CBE^\top\Ocm{\policy{}{E_i}}{i}}_1
	\\
	\text{subject to} & \quad \Ocm{}{i} \in \F{i}  \text{ for all } i \in [N],
	\\
	& \quad \policy{}{c} \in \polspace,
	\\
	& \quad \abs{ \frac{\Ocm{}{i}(\state,\action)}{\sum_{a' \in \Action} \Ocm{}{i}(\state,a')} - \policy{}{c} (\state,\action) } \leq \epsilon\nonumber \\
	&\qquad \quad \text{ for all } i \in [N], \state \in \State,\; \action \in \Action \label{cons:bilinear}.
\end{align}
\end{subequations}
The equivalence here refers to the fact that policies obtained from the optimal occupation measures of the above problem along with the cross-learned policy will be an optimizer of~\eqref{eq:CrossLearning_obj_2}. The constraint~\eqref{cons:bilinear} can be written as $\abs{ \Ocm{}{i}(\state,\action) - \policy{}{c} (\state,\action)\sum_{a' \in \Action} \Ocm{}{i}(\state,a')}  \leq \epsilon \sum_{a' \in \Action} \Ocm{}{i}(\state,a')$ and so it is bilinear in variables $\policy{}{c}$ and $\Ocm{}{i}$. Thus, the feasibility set of the above problem is nonconvex, in general. 
\input{S3_2_McCormick_relax}

%% file: S3_2_McCormick_relax.tex
For computational ease, we use a set of linear inequality constraints to bound the nonconvex feasibility set that is formed by the bilinear constraint~\eqref{cons:bilinear}. To this end, we make use of McCormick envelopes~\cite{GPM:76, FA-JE:83}. Specifically, consider a bilinear constraint $\dummyz = \dummyx \dummyy$ for decision variables $\dummyx$, $\dummyy$, and $\dummyz$ where the former two are further constrained as $\dummyx_{l}\leq \dummyx \leq \dummyx_{u}$ and $\dummyy_{l}\leq \dummyy \leq \dummyy_{u}$. Then, the McCormick envelope for the set
\begin{align}\label{eq:th_mc_nconvex}
    \setdef{(x,y,z) \! \in \! \R{3}}{z \! = \! xy, \dummyx_{l}\leq \dummyx \leq \dummyx_{u}, \dummyy_{l}\leq \dummyy \leq \dummyy_{u}}
\end{align}
is the set consisting of four linear inequalities in place of the bilinear equality: 
\begin{align}
\label{eq:th_mc}
\left\{(x,y,z) \in \R{3}{} \middle\vert \begin{array} {l}
\dummyz \geq \dummyx_{l}y+\dummyx \dummyy_{l} - \dummyx_{l} \dummyy_{l},\\
\dummyz \geq \dummyx_{u}y+\dummyx \dummyy_{u} - \dummyx_{u} \dummyy_{u},\\
\dummyz \leq \dummyx_{u}y+\dummyx \dummyy_{l} - \dummyx_{u} \dummyy_{l},\\
\dummyz \leq \dummyx_{l}y+\dummyx \dummyy_{u} - \dummyx_{l} \dummyy_{u},\\
\dummyx_{l}\leq \dummyx \leq \dummyx_{u}, \dummyy_{l}\leq \dummyy \leq\dummyy_{u}.\end{array}\right\}
\end{align}
In the definition of the above set, the first two inequalities are the so called underestimating convex functions, and the next two are overestimating concave functions. The set defined in~\eqref{eq:th_mc_nconvex} is a subset of that in~\eqref{eq:th_mc}. In the following, we make use of this procedure to form an outer approximation of~\eqref{cons:bilinear}. 

Let $\{\sOcm{}{i},\bipar{}{i}\}_{i \in \until{\Tenv}}$ be the new set of decision variables where $\sOcm{}{i} \in \R{\abs{\State}}{}$ and $\bipar{}{i} \in \R{\abs{\State}\abs{\Action}}{}$. The variable $\sOcm{}{i}(\state)$ will take the value of  $\sum_{\action\in \Action} \Ocm{}{i}(\state,\action)$ and the  variable $\bipar{}{i}(\state,\action)$ will play the role of $\policy{}{c}(\state,\action) \sum_{\action\in \Action} \Ocm{}{i}(\state,\action)$. Then, the bilinear constraint $\bipar{}{i}(\state,\action) = \policy{}{c}(\state,\action) \sum_{\action\in \Action} \Ocm{}{i}(\state,\action)$ will be replaced with four linear inequalities, similar to way explained above. With these additional decision variables, we define the following McCormick relaxation of~\eqref{eq:CrossLearning_obj} as
\begin{subequations}\label{eq:CrossLearning_McCormick}
\begin{flalign}
	\min_{\substack{\{\Ocm{}{i}\}_{i=1}^\Tenv, \policy{}{c}, \\ \{\bipar{}{i}\}_{i=1}^\Tenv, \sOcm{}{i}}} \quad& \sum_{i=1}^{\Tenv} \norm{\CBE^\top \Ocm{}{i} - \CBE^\top \Ocm{\policy{}{E_i}}{i}}_1&&	\\
	\subjectto\quad& \Ocm{}{i}\in\F{i}, \, \, \forall \, i\in\until{\Tenv},&&
	\\
	& \policy{}{c}\in\polspace, \{\bipar{}{i}\}_{i=1}^\Tenv\in\realnonnegative^{\abs{\State}\abs{\Action}}, \, \, \forall \, i\in\until{\Tenv},&&
	\\
	& \{\bipar{}{i}\}_{i=1}^\Tenv\in\realnonnegative^{\abs{\State}\abs{\Action}}, \, \, \forall \, i\in\until{\Tenv},&& \\
	&\sOcm{}{i}(\state)=\sum_{\action\in \Action} \Ocm{}{i}(\state,\action), \, \, \forall \, i\in\until{\Tenv},&&
	\\
	&\text{for all }i \in \until{\Tenv},\; \state\in \State,\;\action\in\Action:&&\nonumber
	\\
	&\, \, \, \, \,\abs{\Ocm{}{i}(\state,\action)-\bipar{}{i}(\state, \action)}\leq \epsilon\sOcm{}{i}(s),&&
	\\
	&\, \, \, \, \, \bipar{}{i}(\state, \action) \geq \Id(\state)\policy{}{c}(\state,\action),\label{mc:c1}&&
	\\
	&\, \, \, \, \, \bipar{}{i}(\state, \action) \geq \sOcm{}{i}(\state) \! + \!  \frac{\abs{\Action}}{1-\dis} (\policy{}{c}(\state,\action) \! - \! 1),\label{mc:c2}&&
	\\
	&\, \, \, \, \, \bipar{}{i}(\state, \action) \leq \sOcm{}{i}(\state) \! + \! \Id(\state)(\policy{}{c}(\state,\action) \! - \! 1), \label{mc:c3}&& 
	\\ 
	&\, \, \, \, \, \bipar{}{i}(\state, \action) \leq \frac{\abs{\Action}}{1-\dis} \policy{}{c}(\state,\action).\label{mc:c4}&&
\end{flalign}
\end{subequations}
Here constraints~\eqref{mc:c1}-\eqref{mc:c4} are obtained using the under-estimators and over-estimators for $\bipar{}{i}(\state, \action) = \policy{}{c}(\state,\action) \sOcm{}{i}(\state)$ along with the bounds $\Id(\state) \leq \sOcm{}{i}(\state) \leq \frac{\abs{\Action}}{1-\dis}$ and $0 \le \policy{}{c}(\state,\action) \leq 1$ for the occupation measure and policy, respectively. The next result summarizes the guarantee of the above approximation.
\begin{proposition}\longthmtitle{Solutions of~\eqref{eq:CrossLearning_McCormick} as approximation of those of~\eqref{eq:CrossLearning_obj_2}}
\label{prop:mc_conv}
If $(\{\policy{}{i}\}_{i=1}^{\Tenv}, \policy{}{c})$ is a feasible point of the CAL problem~\eqref{eq:CrossLearning_obj_2}, then there exist $\{\sOcm{}{i},\bipar{}{i}\}_{i \in \until{\Tenv}}$ such that these variables along with $(\{\Ocm{\policy{}{i}}{i}\}_{i=1}^{\Tenv}, \policy{}{c})$ together are feasible for~\eqref{eq:CrossLearning_McCormick}. Consequently, the optimal value of~\eqref{eq:CrossLearning_McCormick} is a lower bound for the optimal value of~\eqref{eq:CrossLearning_obj_2}. 
\end{proposition}
Note that if $(\{\Ocm{*}{i}\}_{i=1}^{\Tenv},\policy{*}{c})$ is part of the optimizers of~\eqref{eq:CrossLearning_McCormick}, then the obtained policies from these measures might not be feasible for the CAL problem~\eqref{eq:CrossLearning_obj_2}. To obtain feasible policies, one can resort to one of the following two strategies:
\begin{enumerate}
    \item Project all the policies $\{\policy{}{\Ocm{*}{i}}\}_{i=1}^{\Tenv}$ obtained from~\eqref{eq:CrossLearning_McCormick} onto an $\centrality-$ball (under the $\inf$-norm) with its centre as the cross-learned policy $ \policy{*}{c}$.
    \item Project all the individual policies $\{\policy{}{\Ocm{*}{i}}\}_{i=1}^{\Tenv}$ onto an $\centrality$-ball centred at the average policy $\frac{1}{\Tenv}\sum_{i \in \until{\Tenv}} \policy{}{\Ocm{*}{i}} $. 
\end{enumerate}
The former gives more importance to $\policy{*}{c}$, while the later perceives that the obtained individual policies $\policy{}{\Ocm{*}{i}}$ perform well and so their average is selected as an estimate of the cross-learned policy.
\begin{remark}
\longthmtitle{An inner approximation approach}
\label{rm:conv_cons}
{\rm 
The McCormick relaxation described above forms an outer convex approximation of the feasibility set. One can also form an inner convex approximation by using the bound given in Lemma~\ref{lem:map_prop3}. Specifically, given any environment $i \in \until{\Tenv}$ and two occupation measures $\Ocm{}{1},\Ocm{}{2}\in\F{i}$, we have 
\begin{equation}
\label{eq:lem_cc1}
\big\| \policy{}{\Ocm{}{1}}-\policy{}{\Ocm{}{2}} \bigr\|_2 \leq \frac{2}{\Id^{\min}} \norm{\Ocm{}{1}-\Ocm{}{2}}_1,
\end{equation}
where $\Id^{\min} := \min_{\state \in \State} \Id(\state)$ and $\policy{}{\Ocm{}{1}}$ and $\policy{}{\Ocm{}{2}}$ are policies corresponding to measures $\Ocm{}{1}$ and $\Ocm{}{2}$, respectively. This bound was obtained in \cite[Proposition 1]{JZ-AK-ASB-CS-MW:20} and a closer look at the proof in there reveals that occupation measures need not be restricted to $\F{i}$ for the bound to hold. In fact, if two measures $\Ocm{}{1},\Ocm{}{2}$ belong to the set $\JJ$, where
\begin{flalign*}
	\JJ := \setdefbig{\Ocm{}{} \in \realnonnegative^{\abs{\State}\abs{\Action}}}{\Id^{\min} \le \sum_{\action' \in \Action} \Ocm{}{}(\state,\action')  \leq& \frac{\abs{\Action}}{1-\dis} 
	&&\\
	& \text{ for all } \state \in \State},&&
\end{flalign*}
then the bound~\eqref{eq:lem_cc1} is satisfied. Further, for any vector $\dummy \in \R{n}{}$, we have $\norm{\dummy}_\infty \le \norm{\dummy}_2$ and $\norm{\dummy}_1 \le n \norm{\dummy}_\infty$. Using these bounds in~\eqref{eq:lem_cc1}, we obtain
\begin{align*}
   \bigl\| \policy{}{\Ocm{}{1}}-\policy{}{\Ocm{}{2}} \bigr\|_{\infty} \leq \frac{2\abs{\State}\abs{\Action}}{\Id^{\min}} \norm{\Ocm{}{1}-\Ocm{}{2}}_{\infty},
\end{align*}
for all $\Ocm{}{1},\Ocm{}{2} \in \JJ$. Note that $\F{i} \subset \JJ$ for all $i \in \until{\Tenv}$. Using these facts, the convex inner approximation of~\eqref{eq:CrossLearning_obj} is 
\begin{subequations}\label{eq:CrossLearning_relaxed}
	\begin{flalign}
		\min _{\{\Ocm{}{i}\}_{i=1}^\Tenv, \Ocm{}{c}} & \quad  \sum_{i=1}^{\Tenv} \norm{\CBE^\top \Ocm{}{i} - \CBE^\top \Ocm{\policy{}{E_i}}{}}_1&&
		\\
		\text{subject to} & \quad \Ocm{}{i} \in \F{i}, \, \forall i \in \until{\Tenv},&&
		\\
		& \quad \Ocm{}{c} \in \JJ,&&
		\\
		& \quad \norm{\Ocm{}{i}-\Ocm{}{c}}_{\infty}\leq \frac{\Id^{\min}}{2\abs{\State}\abs{\Action}}\centrality, \, \forall i \in \until{\Tenv}.\label{eq:CLR_cons}&&
	\end{flalign}
\end{subequations}
Once an optimizer $(\{\Ocm{*}{i}\}_{i=1}^\Tenv, \Ocm{*}{c})$ of the above problem is obtained, then the individual policies are $\{\policy{}{\Ocm{*}{i}}\}_{i=1}^\Tenv$ and the cross-learned policy is $\policy{}{\Ocm{*}{c}}$. As the size of state and action spaces appear in the denominator of the constraint~\eqref{eq:CLR_cons}, this approximation is very conservative and often leads to infeasibility for large state and action spaces.
}
\oprocend
\end{remark}
\begin{remark}
\longthmtitle{Distributed computation}
{\rm For applications in the real world, we can envision the scenario where information or behavior of the expert is not available at one particular geographical location. For example, two individuals can be driving a two different vehicles in two different geographical locations. In such a case, it is desirable to solve the CAL problem or its convex approximations in a distributed manner. By this we mean that the data about the expert behavior and the model of the environment remains as local information for an agent and is not shared with other agents. Under this information constraint, the convex approximations~\eqref{eq:CrossLearning_McCormick} and~\eqref{eq:CrossLearning_relaxed} both have structures that allow easy implementation of distributed algorithm. They both have objective functions as the summation of local functions and constraints that are local once a consensus constraint is added. For this case, either one can opt for primal-dual distributed algorithms or distributed alternating direction method of multiplier, see~\cite{GN-IN-AC:20} for complete details. However, solving the bilinear problem~\eqref{eq:CrossLearning_obj} in a distributed manner is unexplored in the literature and we plan to pursue it in future.
}
\oprocend
\end{remark}

%% file: S5_Simulations.tex
\section{Simulations}
\label{sec:simulation}

Here we illustrate the properties of the proposed CAL framework using a navigation task in a windy gridworld. Such an environment is often used to demonstrate the efficacy of reinforcement learning algorithms~\cite{RSS-AGB:18}. We consider four gridworlds, each of which consists of  $7\times10$ cells (similar to \cite[Example 6.5]{RSS-AGB:18}), as depicted in Figure~\ref{fig:world}. These four instances differ in the magnitude of the crosswind that is flowing from bottom to top. Each cell in the gridworld is a state of the environment. An agent in the gridworld aims to reach the target cell by taking at each time instance one of the four available actions, i.e., move left, right, up, or down. When the magnitude of the wind at a particular cell is zero, then the action causes intended movement by one unit as long as it respects the boundary. For instance, action up results in moving of the agent by one unit in the upward direction. In case the wind has non-zero magnitude, then the displacement equivalent to the magnitude and along the direction of the wind is added to the displacement caused due to the action of the agent. For example, if the agent opts for moving right and the wind has unit magnitude, then the agent move to the top-right adjacent cell. This specifies completely the transition probability attached to an environment given the wind direction and magnitude at each cell. Roughly speaking, the aim for the agent is to reach the target cell $(3,7)$, see Figure~\ref{fig:world}, from any cell in the gridworld using minimum number of steps.
\begin{figure}[ht]
\centering
\includegraphics[width=8.5cm]{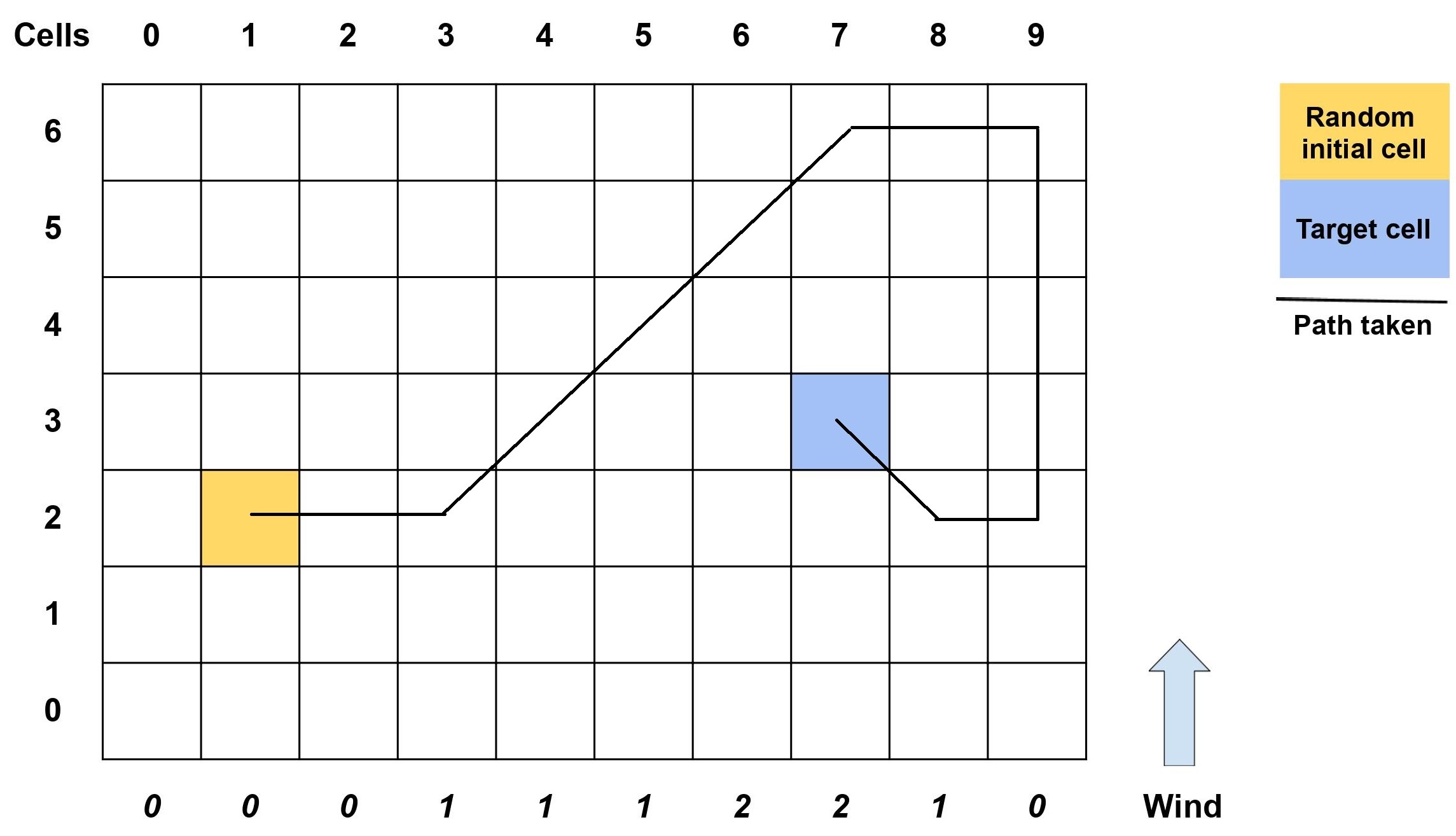}
\caption{An instance of the windy gridworld and a sample trajectory of an agent in it. The yellow and the blue cells are the initial and target cells. The numbers at the bottom of each column stand for the magnitude of wind flowing in the upward direction in all cells belonging to that column.}
\label{fig:world}
\end{figure}
\begin{table*}[ht]
\caption{Number of trajectories out of 200 for which the policies obtained using $\centrality=1$, $\centrality=0.6$, $\centrality=0.2$ and $\centrality=0$ reached the target cell within a maximum of 20 timesteps.}
\centering
\resizebox{2\columnwidth}{!}{\begin{tabular}{c c c c c|c c c c c}
\toprule
Policy & World 1 & World 2 & World 3 & World 4 & Policy & World 1 & World 2 & World 3 & World 4\\
\hline
\hline
 &  & $\centrality=1$& & & & & $\centrality=0.6$ & &\\
\hline
Individual policy 1 & 199 & 186 & 165 & 187  & Individual policy 1 & 197 & 188 & 166 & 198\\
Individual policy 2 & 6 & 198 & 16 & 191 & Individual policy 2 & 13 & 192 & 21 & 166\\
Individual policy 3 & 5 & 2 & 182 & 2 & Individual policy 4 & 25 & 198 & 26 & 200\\
Individual policy 4 & 7 & 196 & 13 & 198 & Individual policy 3 & 31 & 63 & 29 & 43\\
Cross learned policy  & 37 & 160 & 54 & 152 & Cross-learned policy & 45 & 158 & 76 & 127 \\
\hline
&  & $\centrality=0.2$& & & & & $\centrality=0$ & &\\
\hline
Individual policy 1 & 121 & 177 & 123 & 172 &Individual policy 1 & 28 & 57 & 26 & 53\\
Individual policy 2 & 115 & 179 & 123 & 166 & Individual policy 2 & 24 & 58 & 35 & 54\\
Individual policy 3 & 121 & 168 & 131 & 165 & Individual policy 3 & 31 & 63 & 29 & 43\\
Individual policy 4 & 128 & 175 & 122 & 181 & Individual policy 4 & 31 & 61 & 33 & 453\\
Cross learned policy & 123 & 172 & 125 & 156 & Cross-learned policy  & 23 & 58 & 20 & 56\\
\bottomrule
\end{tabular}}
\label{tab:Mc_relax}
\end{table*}

The direction of the wind for all environments and all cells is down to up. For each environment, the magnitude of the wind is same for all cells in one column, refer to Figure~\ref{fig:world}, and so the magnitude for the whole environment is specified by a vector. The wind vectors for four environments are:    
\begin{align*}
    \text{Gridworld 1}&:\;[0\;0\;0\;1\;1\;1\;2\;2\;1\;0]\\
    \text{Gridworld 2}&:\;[1\;1\;0\;0\;0\;2\;0\;0\;1\;0]\\
    \text{Gridworld 3}&:\;[0\;1\;0\;1\;2\;0\;1\;1\;1\;0]\\
    \text{Gridworld 4}&:\;[0\;0\;1\;1\;2\;2\;0\;0\;1\;0]
\end{align*}
As one can observe, there is a commonality to the task specified for each environment, while the transition probabilities are different. To obtain the behavior of the expert specified by the occupation measure generated by the expert, we first obtain expert policies in each environment using $\centrality$-greedy SARSA algorithm, see (\cite[Example 6.5]{RSS-AGB:18}) for further details. Given the expert policies, we compute the discounted occupation measure generated by them using $200$ sample trajectories, each starting randomly at a location in the gridworld and consisting of $100$ time-steps. For cost basis, we assume the simple case of $\abs{\State}\abs{\Action}$ number of vectors given by $\CBe_i = e_i$ for all $i\in\until{\abs{\State}\abs{\Action}}$, where $e_i\in\R{\abs{\State}\abs{\Action}}{}$ has $1$ at the $i^{\text{th}}$ position and all other entries as $0$. This completely specifies the CAL optimization problem that we aim to solve.  We consider four values for the centrality measure, namely $\centrality=1$, $\centrality=0.6$, $\centrality=0.2$ and $\centrality=0$. We employ the McCormick envelope based outer approximation to find an approximate optimizer of the CAL problem. Since the obtained policies might not satisfy the closeness condition~\eqref{eq:CrossLearning_obj_1-cons}, we use the second strategy explained in the discussion following Proposition~\ref{prop:mc_conv} to obtain feasible cross-learned and individual policies.

Table~\ref{tab:Mc_relax}, shows the performance of the obtained policies. For each obtained policy, we compute the number of the times out of randomly generated $200$ trajectories the agent reaches the target state in at most $20$ steps. Higher this number, better is the ability of the agent to steer to the target. One can note that when $\centrality$ is big, the individual policies are close to optimum in their respective environments and their performance in other gridworlds is not necessarily good, see e.g., Individual policy $3$ for $\centrality=1$. On the other extreme is the case of $\centrality = 0$. Here, all policies perform almost similarly across environments. Note that they are not exactly same as the obtained policies are stochastic and we only show success rate for finite number of trajectories. Our presented CAL framework balances both these extreme cases when $\centrality$ is chosen to be between $0$ and $1$. It can be seen that there is a general trend of increase in the number of successes individual policies have in other environments when we move from $\centrality=1$ to $\centrality = 0$. It is surprising that with $\centrality = 0.2$ we obtain policies that outperform policies obtained with $\centrality = 0$ across all environments. This is possibly also due to the fact that our method only finds approximate optimizers. 

%% file: S6_Conclusion.tex
\section{Conclusion}
\label{sec:conc}

We have introduced the cross apprenticeship learning (CAL) framework for apprenticeship learning when the expert trajectories of the task to be learned are available from multiple environments. We presented various properties of the optimizers of the problem that stands at the core of our framework. Further, since the problem is nonconvex, we provided a convex approximation approach to solve it. Our findings were implemented in a numerical example related to navigation in a windy gridword. Future work will explore distributed algorithms for bilinear optimization problems with tunable accuracy so as to solve the CAL problem for a large number of environments. We also wish to study agents' ability to learn from experts in other environments when the number of expert trajectories available is quite different in various environments. Lastly, we would like to explore the scalability of our approach to large-scale state-action spaces.

%% file: main.bbl
\begin{thebibliography}{10}

\bibitem{PA-AC-MQ-AN:06}
P.~Abbeel, A.~Coates, M.~Quigley, and A.~Ng.
\newblock An application of reinforcement learning to aerobatic helicopter
  flight.
\newblock In {\em Advances in Neural Information Processing Systems}, 2006.

\bibitem{PA-DD-AN-ST:08}
P.~Abbeel, D.~Dolgov, A.~Y. Ng, and S.~Thrun.
\newblock Apprenticeship learning for motion planning with application to
  parking lot navigation.
\newblock In {\em 2008 IEEE/RSJ International Conference on Intelligent Robots
  and Systems}, pages 1083--1090, 2008.

\bibitem{PA-AN:04}
P.~Abbeel and A.~Ng.
\newblock Apprenticeship learning via inverse reinforcement learning.
\newblock In {\em Proceedings of International Conference on Machine Learning},
  2004.

\bibitem{FA-JE:83}
F.~A. Al-Khayyal and J.~E. Falk.
\newblock Jointly constrained biconvex programming.
\newblock {\em Mathematics of Operations Research}, 8, 1983.

\bibitem{SB-PK-SM-CT-TZ:21}
S.~Belogolovsky, P.~Korsunsky, S.~Mannor, C.~Tessler, and T.~Zahavy.
\newblock Inverse reinforcement learning in contextual mdps.
\newblock {\em Machine Learning}, 110(9):2295--2334, 2021.

\bibitem{JB-AS:00}
J.~Bonnans and A.~Shapiro.
\newblock {\em Perturbation {A}nalysis of {O}ptimization {P}roblems}.
\newblock Springer, 2000.

\bibitem{PB-DS:19}
P.~Buchholz and D.~Scheftelowitsch.
\newblock Computation of weighted sums of rewards for concurrent {MDP}s.
\newblock {\em Mathematical Methods of Operations Research}, 89(1):1--42, 2019.

\bibitem{JC-JAB-MCF-AR:20}
J.~Cervino, J.~A. Bazerque, M.~C. Fullana, and A.~Ribeiro.
\newblock Multi-task reinforcement learning in reproducing kernel {H}ilbert
  spaces via cross-learning.
\newblock {\em IEEE Transactions on Signal Processing}, 69, 2021.

\bibitem{JC-SH-WJ-MC-SC-YS:22}
J.~Chae, S.~Han, W.~Jung, M.~Cho, S.~Choi, and Y.~Sung.
\newblock Robust imitation learning against variations in environment dynamics,
  2022.

\bibitem{AC-PA-AN:08}
A.~Coates, P.~Abbeel, and A.~Ng.
\newblock Learning for control from multiple demonstrations.
\newblock In {\em Proceedings of the 25th International Conference on Machine
  Learning}, pages 144--151, 01 2008.

\bibitem{PF-ZY-LX-ZF-Zl-DZ:21}
P.~Fang, Z.~Yu, L.~Xiong, Z.~Fu, Z.~Li, and D.~Zeng.
\newblock A maximum entropy inverse reinforcement learning algorithm for
  automatic parking.
\newblock In {\em 2021 5th CAA International Conference on Vehicular Control
  and Intelligence (CVCI)}, pages 1--6, 2021.

\bibitem{AK-GB-JL:19}
A.~Kamoutsi, G.~Banjac, and J.~Lygeros.
\newblock Stochastic convex optimization for provably efficient apprenticeship
  learning.
\newblock In {\em Proceedings of International Conference on Machine Learning},
  2019.

\bibitem{AK-GB-JL:21}
A.~Kamoutsi, G.~Banjac, and J.~Lygeros.
\newblock Efficient performance bounds for primal-dual reinforcement learning
  from demonstrations.
\newblock {\em CoRR}, abs/2112.14004, 2021.

\bibitem{JK-PA-AN:07}
J.~Kolter, P.~Abbeel, and A.~Y. Ng.
\newblock Hierarchical apprenticeship learning with application to quadruped
  locomotion.
\newblock In {\em Advances in Neural Information Processing Systems},
  volume~20, 2007.

\bibitem{GPM:76}
G.~P. McCormick.
\newblock Computability of global solutions to factorable nonconvex programs:
  Part i -- convex underestimating problems.
\newblock {\em Mathematical Programming}, 10:147–175, 1976.

\bibitem{FM-AH-SN-UT:21}
F.~Memarian, A.~Hashemi, S.~Niekum, and U.~Topcu.
\newblock Robust generative adversarial imitation learning via local
  lipschitzness.
\newblock {\em CoRR}, abs/2107.00116, 2021.

\bibitem{GN-IN-AC:20}
G.~Notarstefano, I.~Notarnicola, and A.~Camisa.
\newblock Distributed optimization for smart cyber-physical networks.
\newblock {\em Foundations and Trends in Systems and Control}, 7(3):253--383,
  2019.

\bibitem{TO-JP-GN-JAB-PA-JP:18}
T.~Osa, J.~Pajarinen, G.~Neumann, J.~A. Bagnell, P.~Abbeel, and J.~Peters.
\newblock An algorithmic perspective on imitation learning.
\newblock {\em Foundations and Trends in Robotics}, 7(1-2):1–179, 2018.

\bibitem{MLP:94}
M.~L. Puterman.
\newblock {\em Markov {D}ecision {P}rocesses: {D}iscrete {S}tochastic {D}ynamic
  {P}rogramming}.
\newblock Wiley-Interscience, 1994.

\bibitem{LR-ER:20}
L.~J. Ratliff and E.~Mazumdar.
\newblock Inverse risk-sensitive reinforcement learning.
\newblock {\em IEEE Transactions on Automatic Control}, 65(3):1256--1263, 2020.

\bibitem{RSS-AGB:18}
R.~S. Sutton and A.~G. Barto.
\newblock {\em Reinforcement {L}earning: An {I}ntroduction}.
\newblock The MIT Press, 2018.

\bibitem{US-MB-RS:08}
U.~Syed, M.~Bowling, and R.~Schapire.
\newblock Apprenticeship learning using linear programming.
\newblock In {\em Proceedings of International Conference on Machine Learning},
  pages 1032--1039, 2008.

\bibitem{US-RES:08}
U.~Syed and R.~E. Schapire.
\newblock A game-theoretic approach to apprenticeship learning.
\newblock In {\em Advances in Neural Information Processing Systems},
  volume~20, 2007.

\bibitem{ST-AL-SH:21}
S.~Tesfazgi, A.~Lederer, and S.~Hirche.
\newblock Inverse reinforcement learning: A control lyapunov approach.
\newblock In {\em 2021 60th IEEE Conference on Decision and Control}, pages
  3627--3632, 2021.

\bibitem{ST-AR-TZ-NM:21}
S.~Tu, A.~Robey, T.~Zhang, and N.~Matni.
\newblock On the sample complexity of stability constrained imitation learning,
  2021.

\bibitem{YAY-PLB-XC-AM:19}
Y.~A. Yadkori, P.~L. Bartlett, X.~Chen, and A.~Malek.
\newblock Large-scale markov decision problems via the linear programming dual,
  2019.

\bibitem{HY-PS-MJ-MA:22}
H.~Yin, P.~Seiler, M.~Jin, and M.~Arcak.
\newblock Imitation learning with stability and safety guarantees.
\newblock {\em IEEE Control Systems Letters}, 6:409--414, 2022.

\bibitem{JZ-AK-ASB-CS-MW:20}
J.~Zhang, A.~Koppel, A.~S. Bedi, C.~Szepesvari, and M.~Wang.
\newblock Variational policy gradient method for reinforcement learning with
  general utilities.
\newblock {\em Advances in Neural Information Processing Systems},
  33:4572--4583, 2020.

\bibitem{ZZ-MB-NB:18}
Z.~Zhou, M.~Bloem, and N.~Bambos.
\newblock Infinite time horizon maximum causal entropy inverse reinforcement
  learning.
\newblock {\em IEEE Transactions on Automatic Control}, 63(9):2787--2802, 2018.

\end{thebibliography}
